\documentclass{amsart}
\usepackage[utf8]{inputenc}
\usepackage[hidelinks]{hyperref}
\usepackage{aliascnt}
\usepackage{setspace}
\usepackage{amsmath}
\usepackage{amsthm}
\usepackage{amssymb}
\usepackage{mathtools}
\usepackage{amsfonts}
\usepackage{enumitem}
\usepackage{wasysym}
\usepackage{graphicx}
\usepackage{mathrsfs}
\usepackage[nobysame]{amsrefs}
\usepackage{todonotes}
\usepackage{microtype}

\newtheorem{theorem}{Theorem}
\numberwithin{theorem}{section}
\newaliascnt{lemma}{theorem}
\newtheorem{lemma}[lemma]{Lemma}
\aliascntresetthe{lemma}

\newaliascnt{proposition}{theorem}
\newtheorem{proposition}[proposition]{Proposition}
\aliascntresetthe{proposition}

\newaliascnt{corollary}{theorem}
\newtheorem{corollary}[corollary]{Corollary}
\aliascntresetthe{corollary}

\newaliascnt{conjecture}{theorem}

\aliascntresetthe{conjecture}

\newaliascnt{definition}{theorem}
\newtheorem{definition}[definition]{Definition}
\aliascntresetthe{definition}

\newaliascnt{example}{theorem}

\aliascntresetthe{example}

\newaliascnt{remark}{theorem}
\theoremstyle{remark}

\aliascntresetthe{remark}

\numberwithin{figure}{section}

\newcommand{\NN}{\mathbb{N}}
\newcommand{\ZZ}{\mathbb{Z}}
\newcommand{\QQ}{\mathbb{Q}}
\newcommand{\RR}{\mathbb{R}}
\newcommand{\CC}{\mathbb{C}}
\newcommand{\HH}{\mathbb{H}}

\DeclareMathOperator{\GL}{\operatorname{GL}}
\DeclareMathOperator{\SL}{\operatorname{SL}}
\DeclareMathOperator{\PSL}{\operatorname{PSL}}
\DeclareMathOperator{\SO}{\operatorname{SO}}

\DeclareMathOperator{\Aff}{\operatorname{Aff}}
\DeclareMathOperator{\lcm}{\operatorname{lcm}}

\begin{document}

\title{Computing Periodic Points on Veech Surfaces}

\author[Chowdhury]{Zawad Chowdhury$^1$}
\address{$^1$Massachusetts Institute of Technology}
\email{zawadx@mit.edu}

\author[Everett]{Samuel Everett$^2$}
\address{$^2$University of Chicago}
\email{same@uchicago.edu}

\author[Freedman]{Sam Freedman$^*$}
\address{$^*$Brown University}
\email{sam\_freedman@brown.edu}
\thanks{$^*$ Corresponding author}

\author[Lee]{Destine Lee$^3$}
\address{$^3$Columbia University}
\email{dll2141@columbia.edu}

\date{\today}

\begin{abstract}
A non-square-tiled Veech surface has finitely many periodic points, i.e., points with finite orbit under the affine automorphism group.
We present an algorithm that inputs a non-square-tiled Veech surface and outputs its set of periodic points.
We apply our algorithm to Prym eigenforms in the minimal stratum in genus 3, proving that in low discriminant these surfaces do not have periodic points, except for the fixed points of the Prym involution.
\end{abstract}

\maketitle
\section{Introduction}
A \emph{translation surface} $(X, \omega)$ is a Riemann surface $X$ with a nonzero holomorphic 1-form $\omega$. 
This paper concerns the dynamics of the group $\Aff^+(X, \omega)$ of orientation-preserving affine automorphisms of $(X, \omega)$, specifically the \emph{periodic points} of the $\Aff^+(X, \omega)$-action.
For simplicity of the definitions, we restrict our attention to \emph{Veech surfaces}: translation surfaces where the image of the derivative map $D\colon \Aff^+(X, \omega) \to \SL(2, \RR)$ is a lattice.
Points on a Veech surface $(X, \omega)$ that have a finite $\Aff^+(X, \omega)$-orbit are then called \emph{periodic}.
This definition coincides with the one given by M\"oller in \cite{Moller2006}, and examples include the zeros of the 1-form $\omega$.

Much work has been devoted to the problem of determining periodic points on Veech surfaces, in part due to their role in finite-blocking and illumination problems and providing heuristic evidence for higher-rank orbit closures (see \cite{apisa2020periodic} and the references therein).
Gutkin-Hubert-Schmidt \cite{GutkinHubertSchmidt2003} showed that non-square-tiled Veech surfaces only have finitely many periodic points (this was also established by Möller \cite{Moller2006} and Lanneau-Nguyen-Wright \cite{lanneau_finiteness_2017}, and has been recently recovered as a special case of work of Eskin-Filip-Wright \cite{efw}), and Shinomiya \cite{Shinomiya} established explicit upper bounds on their number.
Periodic points have been explicitly classified in some cases: Möller addressed the case of genus 2 in \cite{Moller2006}; Apisa-Saavedra-Zhang addressed the regular $n$-gon and double regular $n$-gon for $n = 5$, $n \ge 7$ in \cite{apisa2020periodic}; B. Wright addressed the Bouw-Möller examples in \cite{wright2021periodic}.

Nonetheless, the problem of determining the periodic points on an arbitrary Veech surface remains open.
This paper describes a general algorithm that computes the periodic points on a Veech surface:

\begin{theorem}\label{thm:algo}
There is an algorithm that, given a non-square-tiled Veech surface as input, outputs the periodic points on that translation surface.
\end{theorem}

\pagebreak
As a corollary to our proof of the correctness of this algorithm, we obtain:
\begin{corollary}\label{cor:finiteness}
A non-square-tiled Veech surface has finitely many periodic points.
\end{corollary}

The algorithm works as follows.
First, we use the fact that periodic points on Veech surfaces must lie at rational heights in cylinders containing them (see \autoref{lem:rational-height} for a precise statement) to generate finitely many constraints that any periodic point must satisfy.
Next, in \autoref{lem:constraint-reduction} we combine the constraints coming from two cylinder directions to produce a finite collection of line segments containing all of the periodic points. 
Applying a well-chosen element of $\Aff^+(X, \omega)$ to these segments (see \autoref{lem:hyperbolic-without-bad-eigen}) yields a new collection of line segments, also containing all of the periodic points, that intersects the first collection of lines in a finite set of points.
Finally, we use \autoref{lem:reduce-candidate-points} to determine which elements of this finite set have finite orbit under all of $\Aff^+(X, \omega)$, i.e. which points are periodic.

We implemented this algorithm in SageMath \cite{sagemath}, subject to the simplifying assumption that the input translation surface admits a Delaunay triangulation that is both horizontally and vertically periodic.
A novelty of our implementation is a subroutine that uses Delaunay triangulations to efficiently decompose a Veech surface $(X, \omega)$ into cylinders in the direction of any saddle connection. 
In contrast to existing algorithms for performing this computation, our approach does not require searching for saddle connections with a given slope.
Rather, we iteratively contract a given direction until it is sufficiently short, forcing the direction to appear in every triangle of the contracted surface's Delaunay triangulation (see   
\autoref{cylref} and \autoref{sec:cyl-find-alg} for more details).
This implementation has been made publicly available \cite{code}.

Using this implementation, we investigated periodic points on Weierstrass Prym eigenforms in the minimal stratum in genus 3 discovered by McMullen in \cite{McMullenPrym}. (See Section 5 for more information.) These are an infinite family of Veech surfaces whose $\GL^+(2, \RR)$-orbits are specified by an integer discriminant congruent to 0, 1 or 4 modulo 8 (as well as a $\ZZ/2\ZZ$-valued invariant). 
We found:
\begin{theorem}\label{thm:prym}
For Weierstrass Prym eigenforms in genus 3 of nonsquare discriminant $D$ at most 104, the periodic points are the fixed points of the Prym involution.
\end{theorem}
We established this result for discriminants $D \ge 12$ with our implementation applied to the type $\text{A}+$ and type $\text{A}-$ polygonal presentations constructed in \cite{LanneauNguyen} and pictured in Figure \ref{prympoly}.
The remaining case of discriminant $D = 8$ follows from previous work: the corresponding Prym eigenform locus $\Omega E_8(4)$ consists of a single $\GL^+(2, \RR)$-orbit (see \cite{LanneauNguyen}) and contains the Bouw-M\"oller surface $B(3, 4)$, which B. Wright showed in  \cite{wright2021periodic} has only the fixed points of the Prym involution as periodic points.
(After this work, the third author in \cite{Free} proved that \autoref{thm:prym} holds for all nonsquare discriminants.
In fact, the proof relies on our explicit calculations here for the periodic points of Prym eigenforms with small discriminant.)

The paper is organized as follows. In \autoref{sec:background} we give background on translation surfaces, flat geometry, and Delaunay triangulations (which are used in implementing components our algorithm). In \autoref{secAlgo} we prove the correctness of our algorithm, establishing Theorem \ref{thm:algo}. In \autoref{secGen} we detail some of the subroutines used in implementing our algorithm, including the ``cylinder refinement" procedure in \autoref{cylref} mentioned above. Finally, in \autoref{secExp} we present the results of our algorithm applied to Prym eigenforms in the minimal stratum in genus 3, establishing Theorem \ref{thm:prym} which states that their only periodic points are the fixed points of their Prym involution.

\section{Background}\label{sec:background}
We start with a brief background on the theory of translation surfaces.  For an in-depth review see e.g. \cites{wrightsurv, MasurTabach, zorich}. We then state the important Rational Height Lemma (Lemma \ref{lem:rational-height}) that restricts the height of a periodic point in a cylinder. Finally, we give some background on Delaunay triangulations (following Bowman's work in \cite{bowman}), which are used in the implementation of our algorithm.

\subsection{Flat geometry of translation surfaces}
As mentioned above, a translation surface can be succinctly defined as a pair $(X, \omega)$ of a Riemann surface with a choice of nonzero holomorphic 1-form. Alternatively, a translation surface can be described as an equivalence class of a collection of polygons in $\CC$ whose parallel sides are identified by translation. Two translation surfaces are defined to be equivalent if their associated polygonal representations are cut-and-paste equivalent. The elements of the finite $\Sigma \subset X$ of zeros of $\omega$ are known as \emph{singularities} or \emph{cone points}.

A \emph{saddle connection} on a translation surface is a straight line-segment joining two singularities, such that it has no singularities on its interior. Each saddle connection, when viewed on a polygonal presentation of $(X, \omega)$, has two associated vectors $\pm v$ in $\CC$ called its \emph{direction}. Any maximal collection of pairwise disjoint saddle connections determines, a \emph{triangulation} of the translation surface, where any component of the complement of the saddle connections is isometric to the interior of a Euclidean triangle.

A \emph{cylinder} $C$ on a translation surface is the isometric image of an open right angled flat Euclidean cylinder consisting of closed geodesics in a saddle connection direction, whose boundary is a union of saddle connections.  Every cylinder has a \emph{circumference} $c_C$ and a \emph{height} $h_C$, and the ratio $m_C := h_C/c_C$ is the \emph{modulus} of the cylinder. The \emph{direction} of a cylinder is the direction of its boundary saddle connections. We say a translation surface $(X, \omega)$ is \emph{periodic} in a given direction if $(X, \omega)$ can be decomposed as a union of cylinders in that direction, along with their boundaries. 

If $g \in \SL(2, \RR)$ and $(X, \omega)$ is a translation surface given as a collection of polygons, then $g(X, \omega)$ is the translation surface obtained by acting linearly by $g$ on the polygons determining $(X, \omega)$.  Similarly, we can define $g(X, \omega)$ to be the action of taking each chart in the flat structure of $(X, \omega)$ and post-composing it with $g$.

Let $\SL(X, \omega)$ denote the stabilizer of $(X, \omega)$ under the action of $\SL(2, \RR)$.  The \emph{Veech Group} of $(X, \omega)$ is the image of $\SL(X, \omega)$ in $\PSL(2, \RR)$.  If $\SL(X, \omega)$ is a lattice, then $(X, \omega)$ is a \emph{Veech surface}.  

Since $\SL(X, \omega)$ is a discrete subgroup of $\PSL(2, \RR)$, it is a Fuchsian group. The elements of trace 2 in the subgroup are called \emph{parabolic}: they have only one eigendirection, which for a Veech surface corresponds to a cylinder direction. The elements of trace greater than 2 are called \emph{hyperbolic}: they have two eigendirections with reciprocal eigenvalues. For further discussion of Fuchsian groups, see \cite{katok2010fuchsian}.

On a Veech surface, knowing that a surface is horizontally periodic tells us that the Veech group $\SL(X, \omega)$ contains a particular parabolic element:
\begin{proposition}\label{lem:parabolic-matrix-exists}
Let $(X, \omega)$ be a horizontally periodic Veech surface, decomposed into horizontal cylinders $\{C_i\}$.
Then all moduli $m_{C_i}$ have rational ratios, and $\SL(X, \omega)$ contains the matrix $\begin{pmatrix} 1 & t \\ 0 & 1\end{pmatrix}$ where $t := \lcm(m_{C_1}^{-1}, \dots, m_{C_n}^{-1})$.
\end{proposition}
See Lemma 9.7 in \cite{McMullenEigen} for a proof.
Note that a similar result holds for any direction by first rotating the surface.

\begin{definition}[Multiplicity of a cylinder]
Let $(X, \omega)$ be a Veech surface that is horizontally periodic. Consider a horizontal cylinder $H$ in $(X, \omega)$ having modulus $m_H$. Let $t$ be the least common multiple of all of the reciprocals of moduli of horizontal cylinders in the horizontal cylinder decomposition of $(X, \omega)$, as in \autoref{lem:parabolic-matrix-exists}.
We define the \emph{multiplicity} of $H$ in $(X, \omega)$, denoted by $k_H$, to be the integer satisfying the equation $k_H / m_H = t$. 
\end{definition}

For an in-depth review of Veech surfaces, see \cite{HubertSchmidt}.

\subsection{The Rational Height Lemma}


A key property of a periodic point is that its position in any given cylinder containing it is restricted by a certain rationality constraint.
That is, suppose a point $p$ in a cylinder $C$ is at perpendicular distance $d$ from a boundary saddle connection of $C$.
If $C$ has height $h$, then we say $p$ has \emph{rational height} in $C$ if $d / h$ is rational.
The Rational Height Lemma then says that a periodic point must have rational height in every cylinder it lies in:

\begin{lemma}[Rational Height Lemma]\label{lem:rational-height}
Let $\mathcal{C}$ be an equivalence class of cylinders so that any two have a rational ratio of moduli.  If a periodic point belongs to the interior of a cylinder in $\mathcal{C}$ then it lies at rational height.
\end{lemma}
\begin{proof}
For a proof see e.g. Lemma 5.4 in \cite{ApisaRtHt} or Lemma 2.3 in \cite{wright2021periodic}.
\end{proof}

\subsection{Delaunay triangulations}\label{sec:delaunayBackground}

A triangulation $\tau$ of a Euclidean polygon $P$ is \emph{Delaunay} if for each vertex $v$ of $P$ and triangle $T \in \tau$, the interior of the circumcircle of $T$ does not contain $v$.  That no vertex lies inside the circumcircle of a triangle is also called the \emph{Delaunay condition}.

There is also a ``local" characterization of the Delaunay property as follows. Call the quadrilateral $H(e)$ formed by the union of two adjacent triangles $T_1$ and $T_2$ sharing a common edge $e$ a \emph{hinge}. If $\alpha_1$ and $\alpha_2$ are the angles of $T_1$ and $T_2$, respectively, opposite the edge $e$, then let the \emph{dihedral angle} of $e$ be $\alpha(e) \coloneqq \alpha_1 + \alpha_2$. Then, the hinge $H(e)$ is said to be \emph{locally Delaunay} if $\alpha(e) \le \pi$. It is well-known that the triangulation $\tau$ is Delaunay in the first sense if and only if $H(e)$ is locally Delaunay for all edges $e \in \tau$. See Figure \ref{fig:hinges} for examples of hinges. Note that when the dihedral angle of an edge \emph{equals} $\pi$, then the four vertices of the hinge all lie on the (common) circumcircle of $T_1$ and $T_2$. Such hinges are said to be \emph{degenerate}. A triangulation $\tau$ with $\alpha(e) < \pi$ for all edges $e \in \tau$ is said to be a \emph{non-degenerated} Delaunay triangulation, and it is well-known that there is at most one such triangulation.

\begin{figure}
    \centering
    \includegraphics[scale=0.65]{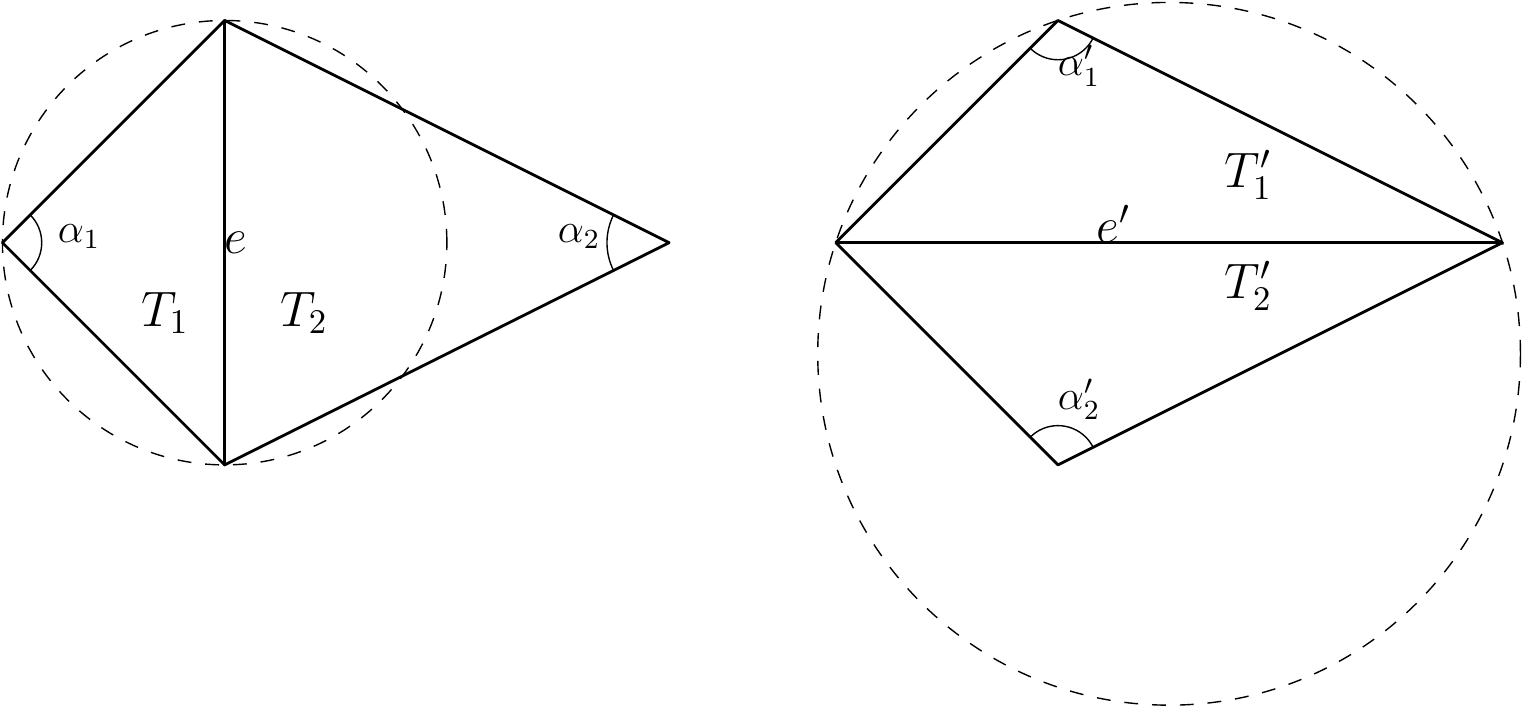}
    \caption{The left hinge is locally Delaunay because $\alpha(e) \le \pi$, whereas the right hinge is not since $\alpha(e') > \pi$. Note that the circumcircle of $T_1'$ contains the vertex of $T_2'$ opposite $e$, so that its triangulation is not Delaunay.}
    \label{fig:hinges}
\end{figure}

The local Delaunay condition extends naturally to a definition of a Delaunay triangulation for a translation surfaces. The existence of a Delaunay triangulation for any translation surface, and their uniqueness for a generic translation surface, were established by Masur and Smillie in \cite{MasSmi}. See also Bowman's exposition in \cite{bowman} and the independent work of Veech in \cite{Veech_F_Strs}.

Given any $g \in \SL(2, \RR)$ and $\tau$ a triangulation of $(X, \omega)$, then $g\cdot \tau = \{g\cdot T | T \in \tau\}$ is a triangulation of $g \cdot (X, \omega)$.  When $\tau$ is a Delaunay triangulation, it is fruitful to consider which matrices $g$ preserve the Delaunay property:

\begin{definition}
    Fix a translation surface $(X, \omega)$, and let $\tau$ be a triangulation of $(X, \omega)$.  The \emph{Iso-Delaunay Region (IDR)} $(S, \tau)$ for $\tau$ is the maximal connected open subset $S \subset \mathbb{H}\cong \SO(2, \RR) \backslash \SL(2, \mathbb{R})$ consisting of right cosets $[g]$ such that $g \circ \tau$ is a Delaunay triangulation of $g \cdot (X, \omega)$.
\end{definition}

With respect to the hyperbolic metric on $\HH$, each IDR is a finite-area, geodesically convex hyperbolic polygon, possible with some vertices at infinity (see Proposition 2.12 in \cite{bowman}). We will call the ideal vertices of an IDR its \emph{cusps}.

Consider the family of triangulations $\{g_{-t} \cdot \tau\}$, where
\[g_t = \begin{pmatrix} e^{t} & 0 \\ 0 & e^{-t} \end{pmatrix}\]
is the Teichmüller geodesic flow. Suppose $\tau$ is a Delaunay triangulation of $(X, \omega)$ such that the corresponding IDR $(S, \tau)$ has a cusp at infinity. Then, it follows that $g_{-t} \cdot \tau$ is a Delaunay triangulation of $g_{-t} \cdot (X, \omega)$ for all $t \ge 0$. In other words, such a translation surface has a Delaunay triangulation that persists under contraction in the horizontal direction and expansion in the vertical direction. Translation surfaces with this property were more amenable to our implementation of our algorithm, motivating the following definition (a version of which appeared in \cite{Veech_F_Strs}):

\begin{definition}
We say a hinge $H$ composed of two triangles in a Delaunay triangulation $\tau$ is \emph{eternally Delaunay} if, for all $t \ge 0$ the transformed hinge $g_{-t} \cdot H$ satisfies the local Delaunay property. If every hinge $H$ in a Delaunay triangulation $\tau$ is eternally Delaunay, then we say the triangulation $\tau$ is \emph{eternally Delaunay}.
\end{definition}
See Section 4 for useful properties of eternally Delaunay triangulations.
We remark that cusps at other boundary points correspond to invariance of the Delaunay property under a rotated $g_t$-flow. For example, if an IDR $(S, \tau)$ for a Delaunay triangulation $\tau$ has a cusp at $0 \in \RR \subset \partial\HH$, then $\tau$ remains Delaunay under $g_{-t}^{-1} \equiv g_t$ flow.

\section{Algorithm for finding periodic points} \label{secAlgo}

In this section, we describe the results which lead up to a proof of Theorem \ref{thm:algo}. First, \autoref{sec:rationality-constraints} describes how we use the Rational Height Lemma (Lemma \ref{lem:rational-height}) to obtain constraints on the coordinates of periodic points. Then \autoref{sec:constraint_reduction} describes an algebraic result called the Constraint Reduction Lemma (\autoref{lem:constraint-reduction}), which describes how constraints can be reduced to a linear equation. With this tool set up, \autoref{sec:finite-set-of-lines} described how we reduce our search space to a finite set of lines. Then, \autoref{sec:seg_to_points} describes how we can reduce our search space further to a finite set of candidate points. Finally \autoref{sec:algo-thm-proof} puts everything together explicitly into one algorithm. Further results pertinent to the implementation of this algorithm are described in Section \ref{secGen}.

\subsection{Rationality Constraints from Cylinders} \label{sec:rationality-constraints}

The Rational Height Lemma (\autoref{lem:rational-height}) tells us that the height of a periodic point inside a cylinder is rational.
Once we fix a coordinate system for points in the cylinder where the origin is in some corner, the height of a point $(x, y)$ will be some linear polynomial $ax + by + c \in K[x, y]$, where $K$ is the field of definition of the translation surface.
Then \autoref{lem:rational-height} specifically tells us that $ax + by + c \in \QQ$ when $(x, y)$ are the coordinates of a periodic point.
We call this a \emph{constraint}, since it constrains the possible coordinates of periodic points.

Our algorithm depends on applying \autoref{lem:rational-height} to two distinct cylinders, so that many constraints might be combined into a linear equation. Thus we apply the lemma to regions, defined as follows:

\begin{definition}[Region]
A region is a connected component of the intersection of two cylinders.
\end{definition}

Without loss of generality, our surface decomposes into horizontal and vertical cylinders (by applying a rotation and shear). Thus we can define regions $R$ on the surface by a horizontal cylinder $H$ and a vertical cylinder $V$. Let $h_H, c_H$ be the height and circumference of $H$, and $h_V, c_V$ the height and circumference of $V$. We embed the region in the Cartesian plane by putting one of the corners at the origin, and drawing the cylinders $H$ and $V$ so that they lie in the first quadrant.

This embedding defines coordinates $(x, y)$ for points in $R$. Suppose $P = (x, y)$ is a periodic point. Then applying \autoref{lem:rational-height} to $H$ and $V$ respectively, we get two constraints for periodic points $(x, y) \in R$:

\[ Q_0 = \frac1{h_H}y \in \QQ; \quad Q_1 = \frac1{h_V}x \in \QQ.\]

We also wish to apply \autoref{lem:rational-height} to a periodic point after the application of a Veech group element, in order to obtain a third constraint.
By \autoref{lem:parabolic-matrix-exists}, there is a parabolic in the Veech group that is a horizontal shear of the form \[T = \begin{pmatrix} 1 & k_H/m_H \\ 0 & 1 \end{pmatrix},\] with $m_H = c_H/h_H$ the modulus and $k_H$ the multiplicity of the cylinder $H$.
Any point in $R$ under the image of $T$ will end up in either $H$, or one of $k_H$ copies of $H$ to the right.
Therefore, in our embedding we draw the entire cylinder $H$, and add $k_H$ copies of $H$ to the right. This defines coordinates for the image $T(x, y)$ for $(x, y) \in R$. See \autoref{fig:cylinder_twists}. 
\begin{figure}
    \centering
    \includegraphics[width=\textwidth]{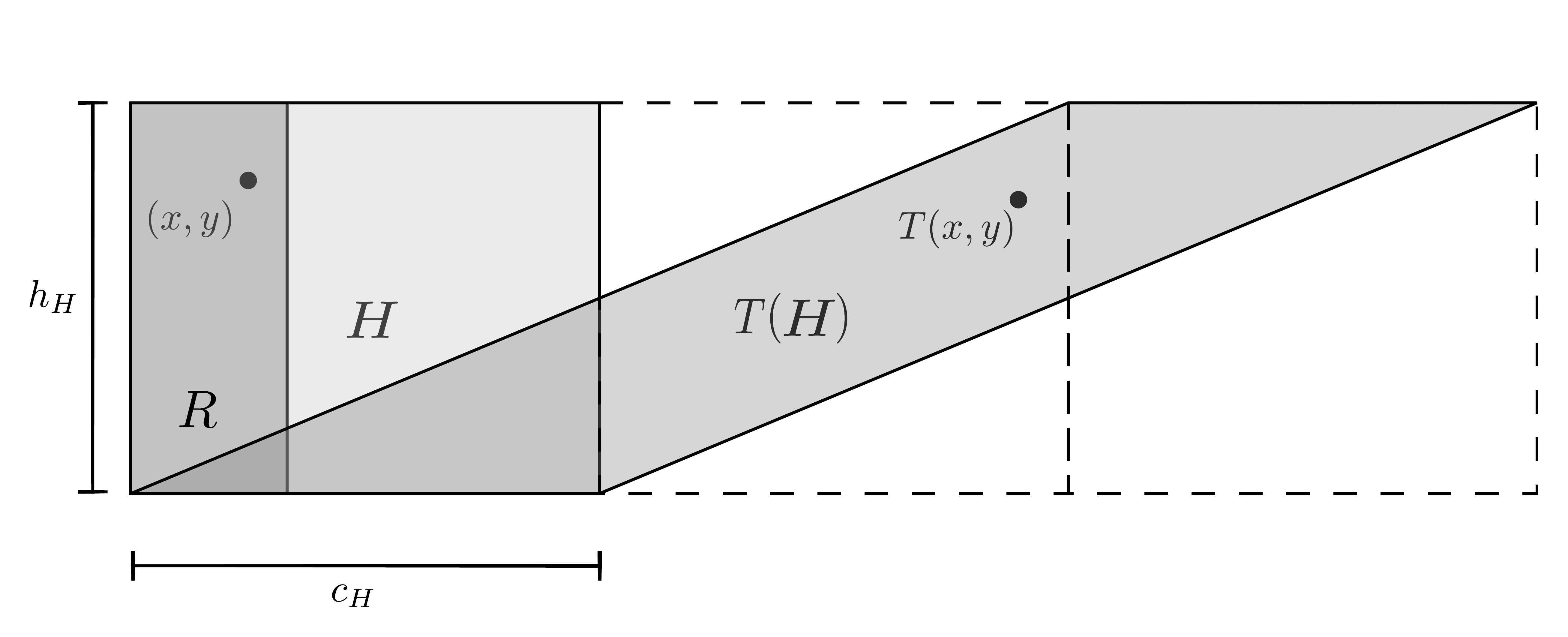}
    \caption{The embedding described above, with cylinder $H$ and it's intersection with a vertical cylinder $R$ labeled. Here the multiplicity $k_H = 2$, so two copies of $H$ are included to the right in the embedding to fit the transformed cylinder $T(H)$.}
    \label{fig:cylinder_twists}
\end{figure}

\pagebreak
Now, consider the $T$ action on the periodic point $(x, y) \in R$.
Suppose this action takes the point to the region $R'$, whose left edge is $d$ apart from the left edge of $R$, as in \autoref{fig:regionshift}.
\begin{figure}
    \centering    \includegraphics[width=\textwidth]{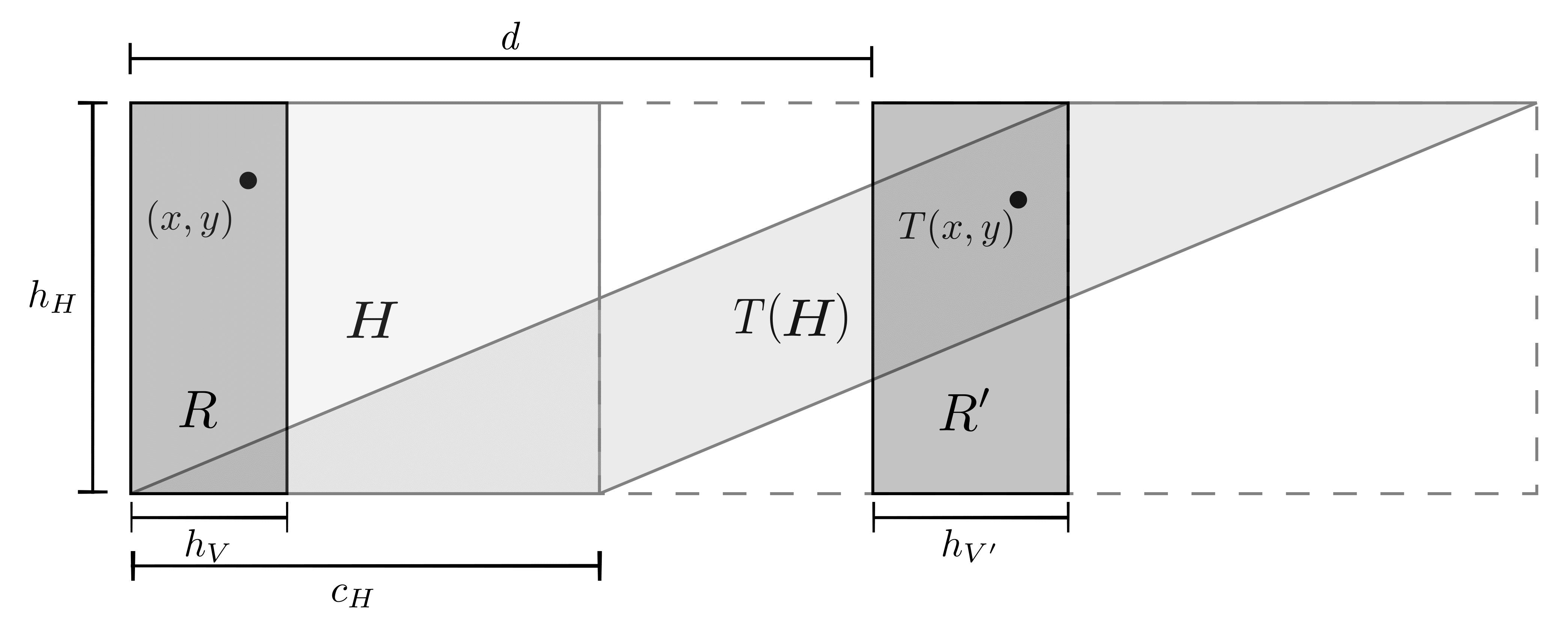}
    \caption{The case where the periodic point $(x, y) \in R$ goes to $T(x, y) \in R'$ under the $T$ action.}
    \label{fig:regionshift}
\end{figure}
If $x'$ is the $x$-coordinate of $T(P)$, the rational height lemma on the vertical cylinder $V'$ of $R'$ gives us

\[\frac1{h_{V'}}(x' - d) \in \QQ. \]

We know $T(x, y) = (x + \frac{k_H}{m_H}y, y)$. Thus we can rewrite the above constraint as
\[ Q_2 = \frac1{h_{V'}}(x + \frac{k_H}{m_H}y - d) \in \QQ.\]

Thus assuming a point is periodic in a region $R$, and that its image under the action of $T$ lies in a region $R'$, provides us with three constraints on its coordinates.

\subsection{The Constraint Reduction Lemma}\label{sec:constraint_reduction}
The previous section shows us how assumptions on the location of a periodic point, or its image under the action of the horizontal shear, can produce constraints on its coordinates. Each constraint produces a measure zero subset of possible periodic points, yet this subset is infinite and dense in every cylinder. To resolve this, we prove an algebraic result which shows how the assumption that three linear functions are rational at a point produces a single linear equation the coordinates of that point must satisfy.

\begin{lemma}[Constraint Reduction Lemma]\label{lem:constraint-reduction}
    Let $K$ be a finite extension of $\QQ$. Take three linear polynomials
    \[
    Q_i(x, y) = a_ix + b_iy + c_i \in K[x, y], \text{ } i=1, 2, 3
    \]
    where $K$ is a nontrivial extension of $\QQ$. Then there exists an algorithm which takes $Q_i$ as input, and outputs a polynomial
    \[Q(x, y) = ax + by + c \in K[x, y]\]
    such that for any $(x, y) \in K^2$, the constraints $Q_i(x, y) \in \QQ$ imply $Q(x, y) = 0$. Moreover, if $Q$ is a constant polynomial, then there exist $d_i \in \QQ$ such that $\sum a_id_i = \sum b_id_i = 0$.
\end{lemma}
\begin{proof}
    We construct the polynomial $Q$ in a sequence of steps, which should illustrate the algorithm.
    
    First, let
    \begin{align*}
        Q_{big}(Z_1, Z_2, Z_3) &= \det\begin{pmatrix} a_1 & b_1 & Z_1 - c_1 \\ a_2 & b_2 & Z_2 - c_2 \\ a_3 & b_3 & Z_3 - c_3\end{pmatrix} \\
        &= m_1Z_1 + m_2Z_2 + m_3Z_3 + d.
    \end{align*}
    
    Here $m_1 = a_2b_3 - a_3b_2$, $m_2 = a_3b_1 - a_1b_3$, $m_3 = a_1b_2 - a_2b_1$ and $d = - \sum c_im_i$.
    
    Now, $K$ is a finite extension of $\QQ$, and thus can be expressed as a $\QQ$-vector field, with $1$ as one of the basis elements. Define $\pi_\QQ: K \to \QQ$ to be the map which projects elements of $K$ to the one dimensional subspace spanned by 1. Now, we define
    \[Q_{rat}(Z_1, Z_2, Z_3) = \pi_{\QQ}(m_1)Z_1 + \pi_{\QQ}(m_2)Z_2 + \pi_{\QQ}(m_3)Z_3 + \pi_{\QQ}(d)\]
    
    Finally, we can define $Q$ to be the polynomial obtained by substituting $Q_i$ for $Z_i$. Thus
    \begin{align*}
        Q(x, y) &= Q_{rat}(Q_1, Q_2, Q_3)(x, y)\\
        &= \pi_{\QQ}(m_1)Q_1(x, y) + \pi_{\QQ}(m_2)Q_2(x, y) + \pi_{\QQ}(m_3)Q_3(x, y) + \pi_{\QQ}(d)\\
        &= \left(\sum_i \pi_{\QQ}(m_i)a_i \right)x + \left(\sum_i \pi_{\QQ}(m_i)b_i \right)y + \left(\sum_i \pi_{\QQ}(m_i)c_i \right) + \pi_{\QQ}(d)\\
        &= \left(\sum_i \pi_{\QQ}(m_i)a_i \right)x + \left(\sum_i \pi_{\QQ}(m_i)b_i \right)y + \left(\sum_i \pi_{\QQ}(m_i)c_i - \pi_{\QQ}(m_ic_i) \right)
    \end{align*}

    We now prove that $Q$ has the properties in the statement of the lemma. Suppose $(s, t)$ is a point such that $Q_i(s, t) \in \QQ$ for all $i$. We shall prove that this implies $Q(s, t) = 0.$

    Let $Q_i(s, t) = r_i$. By the assumption, we know that $r_i \in \QQ$. Now, $Q_{big}(r_1, r_2, r_3)$ evaluates to
    \[ Q_{big}(r_1, r_2, r_3) =  \det\begin{pmatrix}
    a_1 & b_1 & a_1s + b_1t \\
    a_2 & b_2 & a_2s + b_2t \\
    a_3 & b_3 & a_3s + b_3t \\
    \end{pmatrix}.
    \]
    
    Now $s, t \in K$, and thus the third column above is a linear combination of the first two. Thus the determinant evaluates to zero i.e.
    \begin{align*}
        Q_{big}(r_1, r_2, r_3) &= 0\\
        m_1r_1 + m_2r_2 + m_3r_3 + d &= 0
    \end{align*}
    
    We have an element of the $\QQ$-vector field $K$ which equals 0, and thus each projection must equal 0. Also the $r_i$ are rational, which implies
    \begin{align*}
        0 &= \pi_{\QQ}(m_1r_1 + m_2r_2 + m_3r_3 + d) \\
        &= \pi_{\QQ}(m_1)r_1 + \pi_{\QQ}(m_2)r_2 + \pi_{\QQ}(m_3)r_3 + \pi_{\QQ}(d) \\
        &= Q_{rat}(r_1, r_2, r_3)
    \end{align*}

    But $Q_i(s, t) = r_i$, and thus $Q_{rat}(r_1, r_2, r_3) = Q_{rat}(Q_1, Q_2, Q_3)(x, y) = Q(x, y)$. Therefore, $Q_i(s, t) \in \QQ$ for all $i$ implies $Q(s, t) = 0$. 
    
    Now, let $d_i = \pi_{\QQ}(m_i)$, which is necessarily rational. Then if $Q$ is a constant polynomial, the coefficients of $x$ and $y$ are both 0. Thus we have $d_i \in \QQ$ such that $\sum d_ia_i = \sum d_ib_i = 0$.
\end{proof}

Thus we have a way to take three constraints, and produce a linear polynomial. We also know that non-constant linear polynomials define lines as their zero sets (if the variables are treated as coordinates). Thus, for an appropriate choice of coordinates $(x, y)$ on the translation surface, if we obtain three constraints $a_ix + b_iy + c_i = Q_i(x, y) \in \QQ$ for some subset of points (with the added condition that there is no $d_i \in \QQ$ such that $\sum a_id_i = \sum b_id_i = 0$), we would be guaranteed that all those points lie on the line defined by $Q(x, y) = 0$.

\subsection{Producing a Finite Set of Line Segments for Each Region} \label{sec:finite-set-of-lines}

With the machinery of \autoref{lem:constraint-reduction}, we can now turn the constraints defined in \autoref{sec:rationality-constraints} into a line segment.

\begin{lemma} \label{lem:region-constraints}
    Consider a Veech surface $(X, \omega)$. Let $R$ and $R'$ be two regions in the same horizontal cylinder $H$, and on vertical cylinders $V$ and $V'$ respectively. Let $T$ be the horizontal shear in $\SL(X, \omega)$ defined by \autoref{lem:parabolic-matrix-exists}, $c_H$ the circumference of cylinder $H$ and $h_V$ the height of cylinder $V$. Then either $c_H/h_V \in \QQ$, or all periodic points $p = (x, y) \in R$ such that $T(p) \in R'$ lie on some line segment in $R$.
\end{lemma}
\begin{proof}
From \autoref{sec:rationality-constraints}, we have three constraints on any periodic point $(x, y) \in R$ which goes to $R'$ under the $T$ action:

\[ Q_0 = \frac1{h_H}y \in \QQ,\]
\[ Q_1 = \frac1{h_V}x \in \QQ.\]
\[ Q_2 = \frac1{h_{V'}}(x + k_H\frac{c_H}{h_H}y - d) \in \QQ.\]

Suppose $K$ is the field of definition of $(X, \omega)$. We have three polynomials $Q_0, Q_1, Q_2 \in K[x, y]$ such that $Q_i(x, y) \in \QQ$. By \autoref{lem:constraint-reduction}, these produce a linear $Q \in K[x, y]$ such that $Q(x, y) = 0$. Therefore any periodic point $(x, y) \in R$ with $T \cdot (x, y) \in R'$ satisfies $Q(x, y) = 0$. 

We now try to figure out the possibility of $Q$ not representing a line, by applying the degeneracy condition from \autoref{lem:constraint-reduction}. From the polynomials $Q_i$, we have the $a_i$ and $b_i$ (in vector form):
\[ \Vec{a} = \begin{pmatrix} 0 & \frac1{h_V} &  \frac1{h_{V'}} \end{pmatrix}, \quad 
\Vec{b} = \begin{pmatrix} \frac1{h_H} & 0 & \frac{k_H}{h_{V'}}\frac{c_H}{h_H} \end{pmatrix}.\]

Now, if $Q$ is a constant polynomial, we have $d_i \in \QQ$ such that $\sum d_ia_i = \sum d_ib_i = 0$. Plugging in the $a_i$, we would have $d_1/h_V + d_2/h_{V'} = 0$ or $h_V/h_{V'} = - d_1/d_2 \in \QQ$. Plugging in the $b_i$, we get $d_0/h_H + d_2\frac{k_H}{h_{V'}}\frac{c_H}{h_H} = 0$ or $c_H/h_{V'} = - \frac{d_0}{k_Hd_2} \in \QQ$. Combined, these imply $\frac{c_H}{h_V} \in \QQ$. 

Thus if we have $\frac{c_H}{h_V} \notin \QQ$, then the polynomial $Q$ defines a line segment by \autoref{lem:constraint-reduction}. In that case, all periodic points such that $p \in R$ and $T(p) \in R'$ must be on the line segment $\{(x, y) \in R: Q(x, y) = 0.\}$
\end{proof}

With this lemma, for each region we can construct the finite set of segments on which periodic points must lie.

\begin{lemma} \label{lem:region-candidate-segments}
    Consider a non-square-tiled Veech surface $(X, \omega)$. For each region $R$ on the surface, there exists a finite set of line segments $L$ such that all periodic points in $R$ lie on some line segment $\ell \in L$.
\end{lemma}
\begin{proof}
    Without loss of generality, $R$ is defined by a horizontal and vertical cylinder $H, V$. Let $h_H, c_H$ be the height and circumference of $H$, and $h_V, c_V$ the height and circumference of $V$. Recall that there are nonzero rationals $r, s \in \QQ$ such that the parabolic elements $P_H$ and $P_V$ determined by the horizontal and vertical cylinder decompositions, respectively, are of the form
    $$P_H = \begin{pmatrix}
    1 & r c_H / h_H \\ 0 & 1
    \end{pmatrix}, \qquad P_V = \begin{pmatrix}
    1 & 0 \\ s c_V / h_V & 1
    \end{pmatrix}. $$
    
    Now, since $(X, \omega)$ is not square tiled, its trace field cannot be $\QQ$ (see \cite{GutkinJudge}). By Claim 2.1 in \cite{Hubert2006}, this implies that
    $$(r c_H / h_H) \cdot (s c_V / h_H) \notin \QQ,$$
    or equivalently $\frac{c_Vc_H}{h_Hh_V} \notin \QQ$. Then we have $\frac{c_V}{h_H} \notin \QQ$, or $\frac{c_H}{h_V} \notin \QQ$. Without loss of generality, assume the former is true. Then by \autoref{lem:region-constraints}, for any region $R' \in H$, the periodic points $p \in R, T(p) \in R'$ lie on a line segment ($T$ is the horizontal shear). Iterating over all of the finitely many regions in $H$ (taking $k_H$ extra copies to account for its image under $T$, where $k_H$ is the multiplicity of $H$), we have a finite set of line segments covering all periodic points in $R$.
\end{proof}

\subsection{From Segments To Points}\label{sec:seg_to_points}

We now have a finite set $S$ of line segments on which any periodic points must lie. If we apply Veech group elements to these segments, the periodic points go to periodic points. Thus any periodic point must also lie on some segment in $g \cdot S$. Therefore if we can pick $g \in \SL(X, \omega)$ such that $g\cdot S \cap S$ is a finite set of points, we will have reduced our candidate line segments to a finite set of candidate points. Finding such a $g$ involves the following lemmas:

\begin{lemma}\label{lem:hyperbolic-without-bad-eigen}
Given a finite set of line segments $S$ on a translation surface $(X, \omega)$, there exists a hyperbolic element in the Veech group $g_0 \in \SL(X, \omega)$ such that none of the line segments in $S$ are parallel to eigenvectors of $g_0$.
\end{lemma}
\begin{proof}
Without loss of generality, our translation surface is periodic in both horizontal and vertical directions (up to a rotation and a shear). Then, \autoref{lem:parabolic-matrix-exists} gives us a horizontal shear $M_H$ and a vertical shear $M_V$ in $\SL(X, \omega)$. Let the nonzero off diagonal elements of these matrices be $t_H$ and $t_V$ respectively. Consider the matrix $M_H^aM_V^a$ for $a\in\NN$, which is always a hyperbolic element. An explicit calculation shows that the slopes of the eigenvectors of this matrix are
\[\frac12(at_H \pm \sqrt{a^2t_H^2 + 4t_H/t_V}).\]
We can calculate that the product of the two eigenvector slopes is $-t_H/t_V$. As $a$ increases, the absolute value positive eigenvector slope strictly increases and the absolute value of the negative eigenvector slope strictly decreases. Therefore only finitely many $a$ lead to eigenvector slopes which are parallel to segments in $S$. Picking a different value of $a$, we obtain a hyperbolic element of the Veech group $g_0$ such that none of its eigenvectors are parallel to segments in $S$.
\end{proof}

\begin{lemma} \label{lem:lines-to-points}
Consider a hyperbolic $h \in \SL(X, \omega)$, and a finite set $S$ of line segments none of which are parallel to an eigenvector of $h$. Then there exists $n$ such that $h^n \cdot S \cap S$ is a finite set of points.
\end{lemma}
\begin{proof}
Let $\theta(v_1, v_2)$ be the angle in $[-\pi/2, \pi/2]$ between vectors $v_1, v_2$, defined by 

\[\theta(v_1, v_2) = cos^{-1}\left(\frac{v_1\cdot v_2}{|v_1||v_2|}.\right)\]

The angle between a segment and a vector can be calculating by taking the vector parallel to the segment. This angle is well defined modulo $\pi$.

Consider the attracting eigenvector $v_a$ of $h$, which exists as it a hyperbolic element of the Veech group. For a given $\epsilon > 0$ and $v$ not parallel to eigenvectors of $h$, there exists $n = n(\epsilon, v)$ such that $|\theta(v_a, h^n\cdot v)| < \epsilon$.

By assumption, in our set $S$ none of the line segments are parallel to $v_a$. Set $\epsilon = \frac12\min_{s\in S} \theta(v_a, s)$. Now, pick
\[n = \max_{s\in S} n(\epsilon, s).\]

Thus $|\theta(v_a, h^n\cdot s)| < \epsilon$ for all $s\in S$. So any segment in $h^n \cdot S$ has a different slope from all the segments in $S$, and so their overlap can only be a set of points. 

For a particular $n$, we have a bound on the length of segments in $h^n \cdot S$ and thus a bound on the number of possible intersections between a segment in $h^n \cdot S$ and a segment in $S$. Therefore the intersection $h^n \cdot S \cap S$ will be a finite set of points.
\end{proof}

Therefore by picking $h$ from \autoref{lem:hyperbolic-without-bad-eigen} and $n$ from \autoref{lem:lines-to-points}, we obtain $g = h^n$ such that $g \cdot S \cap  S$ is a finite set of points.

\subsection{Reducing Candidate Points to Periodic Points}

We are very close to obtaining the exact set of periodic points on a translation surface. Before we can prove \autoref{thm:algo}, we must describe one last subroutine:

\begin{lemma} \label{lem:reduce-candidate-points}
Consider a translation surface $(X, \omega)$ and a finite set $S$ which contains all of its periodic points. There exists an algorithm which takes the generators of $\SL(X, \omega)$ as input and outputs the set of periodic points of the surface.
\end{lemma}
\begin{proof}
Suppose $p_1, \dots p_n$ are the points in $S$, and $g_1,\dots, g_m$ are the generators of $\SL(X, \omega)$. Note that since $\SL(X, \omega)$ is a lattice, it is finitely generated. The algorithm then proceeds as follows:
\begin{enumerate}
    \item Create a graph with a vertex $v_i$ for each $p_i \in S$, as well as a special vertex $v_X$
    \item For each point $p_i$ and generator $g_k$, calculate the point $p' = g_k \cdot p_i$. If $p' = p_j \in S$, then draw an edge between $v_i$ and $v_j$. Otherwise if $p' \notin S$, draw an edge between $v_i$ and $v_X$.
    \item Find the vertices which are not in the connected component containing $v_X$. Then the points corresponding to those vertices are our periodic points.
\end{enumerate}

If we consider a connected component without $v_X$, then the points in that component get permuted by any generator $g_k \in \SL(X, \omega)$. Thus for any point in that component, the component is its $\SL(X, \omega)$ orbit. This is finite, therefore every point which we include in our output is a periodic point.

Now consider a point $p_i$ in the connected component with $v_X$. Following a path from $v_i$ to $v_X$, we can construct an element of $\SL(X, \omega)$ which sends $p_i$ to a point outside of $S$, which must be non-periodic. The $\SL(X, \omega)$ orbit of a periodic point consists of periodic points, thus $p_i$ cannot be periodic. Since $S$ contains all the periodic points, taking the points in the connected components not containing $v_X$ precisely outputs the set of periodic points. 
\end{proof}

\subsection{Proof of \autoref{thm:algo} and \autoref{cor:finiteness}} \label{sec:algo-thm-proof}

We are finally ready to prove \autoref{thm:algo}. We restate the theorem with specifics:

\begin{theorem}[\autoref{thm:algo}]
Suppose $(X, \omega)$ is a non-square-tiled Veech surface.  There is an algorithm that takes $(X, \omega)$ as input and outputs the periodic points on the surface.
\end{theorem}
\begin{proof}
The algorithm has the following steps, whose correctness can be verified by the lemmas in the previous sections:

\begin{enumerate}
    \item Obtain the cylinder decomposition of the surface in two distinct periodic directions. Without loss of generality, by applying shears to the surface we may assume these two periodic directions are the horizontal and vertical directions. These cylinders partition the surface into connected components of intersections of a horizontal cylinder with a vertical cylinder. Call these regions $R_1, R_2, \dots R_k$. In practice, \autoref{cylref} proves that these regions can be obtained through a computationally tractable triangulation, and \autoref{sec:cyl-find-alg} describes the algorithm to obtain a cylinder decomposition producing such regions.
    \item For each region $R_i$, obtain the set of line segments $L_i$ output by \autoref{lem:region-candidate-segments}. We know that all periodic points on $R_i$ lie on some line segment $\ell \in L_i$. Take $S = \bigcup_{i=1}^k L_i$. Then any periodic point on $(X, \omega)$ lies on some segment in $S$.
    \item By Lemma \ref{lem:lines-to-points}, find some $g \in \SL(X, \omega)$ such that $g\cdot S \cap S = P$ is a finite set of points. Any periodic point in $S$ also lies on some segment in $g\cdot S$. Thus $P$ is a finite set of points that contains all the periodic points of $(X, \omega)$.
    \item Apply the algorithm described in \autoref{lem:reduce-candidate-points} on the set $P$, to obtain the set of points periodic under the entire 
    Veech group.
\end{enumerate}
\end{proof}

As a corollary of the previous lemmas, we obtain that a non-square-tiled Veech surface has finitely many periodic points: 

\begin{proof}[Proof of \autoref{cor:finiteness}]
By applying \autoref{lem:region-candidate-segments} to the finitely many regions determined by a horizontal and vertical cylinder decomposition, we obtain a finite set $S$ of line segments containing all of the periodic points. Applying \autoref{lem:lines-to-points} with the hyperbolic element constructed in \autoref{lem:hyperbolic-without-bad-eigen} to this set $S$, we obtain a finite set of points containing all of the periodic points.
\end{proof}

\section{Results on Delaunay triangulations of translation surfaces}\label{secGen}

In order to implement the algorithm detailed in the proof of Theorem \ref{thm:algo}, we require a representation of translation surfaces that lends itself to performing the various steps of the algorithm computationally.  We found Delaunay triangulations of translation surfaces, as detailed in Section \ref{sec:background}, to be the ideal representation.  In this section, we detail results pertaining to Delaunay triangulations of translation surfaces, that are necessary for ensuring an implementation of the algorithm that behaves correctly.  In particular, we explain the details of how we find cylinders as in Step (1) of the algorithm given in proof of Theorem \ref{thm:algo}.

\subsection{Cylinder Refinement}\label{sec:cylinder_refinement}
Call a collection $\{T_i\}$ of triangles in a triangulation of translation surface $(X, \omega)$ a \emph{refinement} of a cylinder $C$, if $C = \bigsqcup T_i$ (see Figure \ref{fig:cylinderRef} for an example).  In addition, we say a triangulation $\tau$ of a translation surface \textit{comes from a Delaunay triangulation $\bar{\tau}$} of a different translation surface, if $\tau = M\cdot \bar{\tau}$ for some $M \in SL(2, \mathbb{R})$.  

We begin by stating the following ``cylinder refinement proposition," that is used in a number of steps of our implementation of the algorithm given in Theorem \ref{thm:algo}, as in determining intersection regions of horizontal and vertical cylinders of the surface.  

\begin{proposition}[Cylinder Refinement]\label{cylref}
Let $(X, \omega)$ be a Veech surface, periodic in direction $v$.  Then there exists a triangulation $\tau$ of $(X, \omega)$, coming from a Delaunay triangulation, such that every triangle $T_i \in \tau$ has an edge parallel to $v$.
\end{proposition}

\begin{corollary}\label{cor:refinement}
Let $(X, \omega)$ be a Veech surface, periodic in direction $v$ with cylinder decomposition $\{C_i\}$ in direction $v$.  Then there exists a triangulation $\tau$ of $(X, \omega)$, coming from a Delaunay triangulation, such that for each cylinder $C_i$, there is a collection of triangles in $\tau$ that form a refinement of $C_i$.
\end{corollary}
\begin{proof}
Let $(X, \omega)$ be a Veech surface periodic in direction $v$, with triangulation $\tau$ such that every triangle $T_i \in \tau$ has an edge parallel to $v$ as by Proposition \ref{cylref}.  We construct a refinement of a cylinder as follows.  Take any triangle $T_i \in \tau$, and inductively develop out across edges not parallel to $v$ to obtain a collection of triangles $\{T_i\}$, where each triangle in the collection is glued to exactly two other triangles also in the collection, along edges not parallel to $v$.  No triangle has a cone point along the interior of an edge by definition, and hence the collection of triangles must compose a cylinder, with the boundary of the cylinder given by the collection of edges parallel to $v$.  See Figure \ref{fig:cylinderRef} for an example.
\end{proof}

\begin{figure}
    \centering
    \includegraphics[scale=.2]{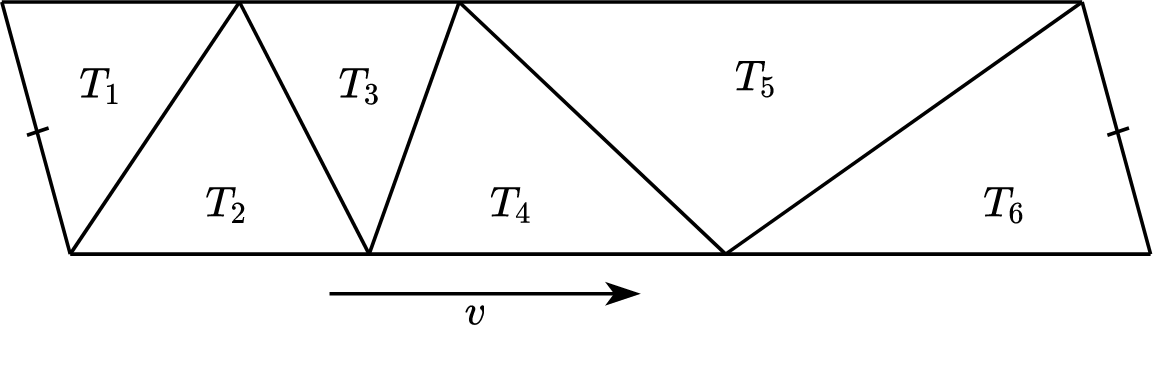}
    \caption{A refinement of a cylinder by a collection of triangles with edges parallel to direction $v$.}
    \label{fig:cylinderRef}
\end{figure}

We require the following lemma before proving Proposition \ref{cylref}.

\begin{lemma}\label{lem:band}
Let $O = (0, 0)$, $P_1 = (x_1, y_1)$, $P_2 = (x_2, y_2)$, $P_3 = (x_3, y_3)$ be the coordinates of a hinge $H$ formed by triangles $T_1 = \bigtriangleup OP_2P_3$ and $T_2 = \bigtriangleup OP_1P_2$ as in Figure \ref{fig:nonED}. Moreover, assume $0 \leq y_1 \leq y_2$, and $0 < y_3 < y_2$. Then $H$ is not eternally Delaunay.
\end{lemma}
\begin{proof}
We prove the lemma in two main steps, first assuming $0 < y_1 < y_2$, and second considering the case when $y_1 = 0$ or $y_2$.   In the first case, the standing assumption $0 < y_3 < y_2$ precludes any of the triangles composing the hinge from having a horizontal edge: $T_1$ and $T_2$ are joined along edge $\overline{OP_2}$, and neither $y_1$ nor $y_3$ can equal $0$ or $y_2$.

Let $\theta_1 = \angle P_1OP_3$, and $\theta_2 = \angle P_1P_2P_3$.
When $0 < y_1 < y_2$, it is sufficient to show that under the application of $g_{-t}$ to $H$, both $\cos(\theta_1) \to 1$ and $\cos(\theta_2) \to 1$ as $t \to \infty$. 
For then it would follow that $\theta_1$ and $\theta_2$ approach 0 when $g_{-t}$ is applied to hinge $H$ as $t \rightarrow \infty$.
This implies the dihedral angle $\alpha(e) \rightarrow 2\pi$ and hence $H$ is not eternally Delaunay (see Section \ref{sec:delaunayBackground}).
We demonstrate this fact for $\theta_1$, with the computation for $\theta_2$ being similar.  For $\theta_1 = \angle P_1OP_3$ on the sheared hinge $g_{-t} H$, we have 
\begin{align*}
  \cos(\theta_1) = \frac{g_{-t}\overrightarrow{P_3}\cdot g_{-t}\overrightarrow{P_1}}{||g_{-t}\overrightarrow{P_3}||_2 ||g_{-t}\overrightarrow{P_1}||_2} &= 
    \frac{e^{-2t}x_1x_3+e^{2t}y_1y_3}{\sqrt{(e^{-2t}x_1^2+e^{2t}y_1^2)(e^{-2t}x_3^2+e^{2t}y_3^2)}} \\
    &= \frac{e^{-2t}x_1x_3+e^{2t}y_1y_3}{\sqrt{e^{-4t}x_1^2x_3^2+x_1^2y_3^2+x_3^2y_1^2+e^{4t}y_1^2y_3^2}}
\end{align*}
Taking the limit of the result as $t\rightarrow \infty$ gives $1$, implying $\theta_1 \rightarrow 0$ as $t \rightarrow \infty$.

We now consider the case when $y_1 = 0$.  Then $\cos(\theta_1)$ is given by
\begin{align*}
  \cos(\theta_1) = \frac{g_{-t}\overrightarrow{P_3}\cdot g_{-t}\overrightarrow{P_1}}{||g_{-t}\overrightarrow{P_3}||_2 ||g_{-t}\overrightarrow{P_1}||_2} &= 
    \frac{e^{-2t}x_1x_3}{e^{-t}x_1\sqrt{e^{-2t}x_3^2+e^{2t}y_3^2}} \\
    &= \frac{e^{-t}x_3}{\sqrt{e^{-2t}(x_3^2+e^{4t}y_3^2)}} \\
    &= \frac{x_3}{\sqrt{x_3^2+e^{4t}y_3^2}}
\end{align*}
But taking the limit of the result as $t\rightarrow \infty$ gives 0, which implies $\theta_1$ approaches $\pi/2$ under $g_{-t}$-flow, since $\cos(\theta_1) \rightarrow 0$ as $t\rightarrow \infty$.  Similar computation shows the angle $\theta_2$ between vectors $\overrightarrow{P_1-P_2}$ and $\overrightarrow{P_3-P_2}$ approaches $0$ as the hinge $H$ is transformed under $g_{-t}$-flow and $y_1=0$.  Then the dihedral angle $\alpha(e) \rightarrow 3\pi/2$ as $t\rightarrow \infty$, implying the hinge $H$ is not eternally Delaunay when $y_1 = 0$.  The case of $y_1 = y_2$ is a similar computation.  
\end{proof}

\begin{figure}
    \centering
    \includegraphics[scale=.2]{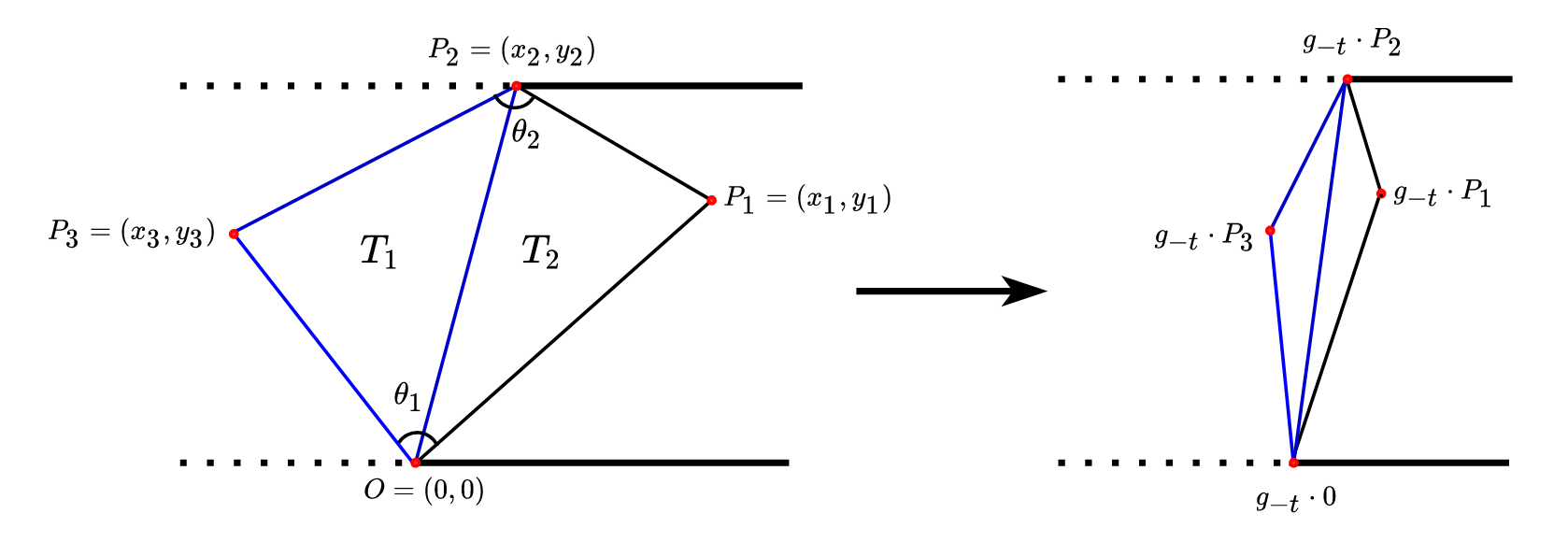}
    \caption{A hinge $H$ which is not eternally Delaunay.}
    \label{fig:nonED}
\end{figure}

We are now in a position to prove Proposition \ref{cylref}.  

\begin{proof}[Proof of Proposition \ref{cylref}]
Any Veech surface $(X, \omega)$ periodic in direction $v$ can be made horizontally periodic by a rotation.  Recall that IDRs divide the fundamental domain into finitely many polygons.
As a consequence, the fundamental domain of this rotated surface will have an IDR with a cusp at infinity, and hence the surface has a triangulation associated to a noncompact IDR.
As such, upon rotation and application of $g_{-t}$ flow for sufficiently large $t$, we may obtain a surface with a Delaunay triangulation associated with such a noncompact IDR, where it follows that every hinge in such a triangulation is eternally Delaunay.
It then suffices to show that any triangulation $\tau$ associated with such a non-compact IDR satisfies the property that every triangle $T_i \in \tau$ has a horizontal edge.

Let $\tau$ be a triangulation associated with a non-compact IDR. 
We aim to show every triangle $T_i \in \tau$ has a horizontal edge.
Proceeding by contradiction, assume there exists a triangle $T_1 \in \tau$ with no horizontal edge, yet every hinge in the triangulation is eternally Delaunay since the triangulation is associated with a non-compact IDR.
Let points $O, P_1, P_2, P_3$ be as in Lemma \ref{lem:band} and Figure \ref{fig:nonED}.
Let $T_1 = \bigtriangleup OP_2P_3$ so that without loss of generality the coordinates $O = (0, 0)$, $P_2 = (x_2, y_2)$, $P_3 = (x_3, y_3)$ satisfy $0 < y_3 < y_2$.

Let $e_1$ denote the edge $\overline{OP_2}$ of $T_1$, and let $T_2 = \bigtriangleup OP_1P_2$ denote the triangle joined to $T_1$ across edge $e_1$, as in Figure \ref{fig:nonED}.
By Lemma \ref{lem:band}, if $0 \leq y_1 \leq y_2$, then the hinge $H$ generated by $T_1$ and $T_2$ would not eternally Delaunay, and could not be in the triangulation.
Hence, it is necessary that $y_1 \in (y_2, \infty) \cup (0, -\infty)$.
Without loss of generality, let $y_1 \in (y_2, \infty)$.
Then it follows that the edge $e_2 = \overline{OP_1}$ of triangle $T_2$ is such that the $y$-coordinate of $e_2$ is strictly less than the $y$-coordinate of $e_1$.

Repeating the process inductively, for $i \ge 1$, we obtain a sequence of triangles $\{T_i\}$ with non-horizontal edges $\{e_i\}$ whose $y$-coordinates strictly decrease.  But any triangulation of $(X, \omega)$ has a finite number of edges, so the $y$-coordinates cannot strictly increase forever, implying existence of a hinge that is not eternally Delaunay by Lemma \ref{lem:band}, a contradiction.  We conclude that every triangle in $\tau$ has a horizontal edge, which completes the proof.
\end{proof}

\subsection{Cylinder Finding Algorithm}\label{sec:cyl-find-alg}

We now describe an algorithm that takes a Veech surface $(X, \omega)$ and a periodic direction of the surface, and returns a triangulation that comes from a Delaunay triangulation, which forms a refinement of the cylinder decomposition of the surface in the specified periodic direction.

\begin{definition}
Let $H$ be a hinge quadrilateral defined by two adjacent triangles $T_1$ and $T_2$ in a triangulation.  A \emph{hinge flip} is the operation by which $H$ is replaced by the new hinge $H'$ formed by the other diagonal in the quadrilateral. See \autoref{fig:hinges} for an example.
\end{definition}

As remarked in Section 2.3, if a hinge $H$ is not locally Delaunay, the flipped hinge $H'$ is locally Delaunay.
Furthermore, if a triangulation is not Delaunay, it can be made Delaunay by flipping a finite number of hinges.
To this end, we can make any triangulation of a Veech surface Delaunay via a finite sequence of hinge flips; the associated IDR to this triangulation will be of particular interest in the sequel. 

The following algorithm computes a cylinder refinement of a Veech surface.
\begin{enumerate}
    \item Begin with a Delaunay triangulation $\tau$ of a Veech surface $(X, \omega)$, and a cylinder direction $v \in S^1$. 
    \item If $v = (0, 1 )^T$, then apply a $\pi/2$-rotation to $\tau$.
    Otherwise, if $v = (x, y)^T$, apply the shear
    \[
    M = \begin{pmatrix}
    1 & 0 \\
    -\frac{y}{x} & 1
    \end{pmatrix}
    \]
    to the triangulation $\tau$, transforming to a horizontal direction.  Let $\tau'$ denote the resulting normalized triangulation.
    \item Apply the $g_{-t}$ flow matrix to $\tau'$, keeping the triangulation Delaunay by hinge flips, until every triangle in the triangulation has an edge parallel to $(1, 0)^T$.  This process must terminate because there exists finite time $t$ such that $g_{-t}\cdot \tau'$ corresponds to a noncompact IDR after making the triangulation Delaunay, and by the proof of \autoref{cylref} every triangle in such a triangulation must have a horizontal edge. Denote the resulting triangulation $\overline{\tau}$.
    \item Assemble the collection of triangles in $\overline{\tau}$ into cylinders by developing across non-horizontal edges.  Such a refinement of cylinders by triangles in $\overline{\tau}$ is guaranteed by the proof of \autoref{cor:refinement}.
    \item Apply inverse $g_{t}$-flow to $\overline{\tau}$ for the same time used in step (3), and then apply the inverse shear or $\pi/2$ rotation matrix from step (2) to $g_t \cdot \overline{\tau}$.
\end{enumerate}
The result is then a triangulation of $(X, \omega)$ that comes from a Delaunay triangulation, so that the triangulation gives a refinement of the cylinders composing the cylinder decomposition of the surface in direction $v$.

\section{Experimental results} \label{secExp}

In this section we apply our algorithm to different Veech surfaces to compute their periodic points.

\subsection{Eigenforms in \texorpdfstring{$\mathcal{H}(2)$}{}}

\begin{figure}[b]
    \centering
    \includegraphics[scale=.15]{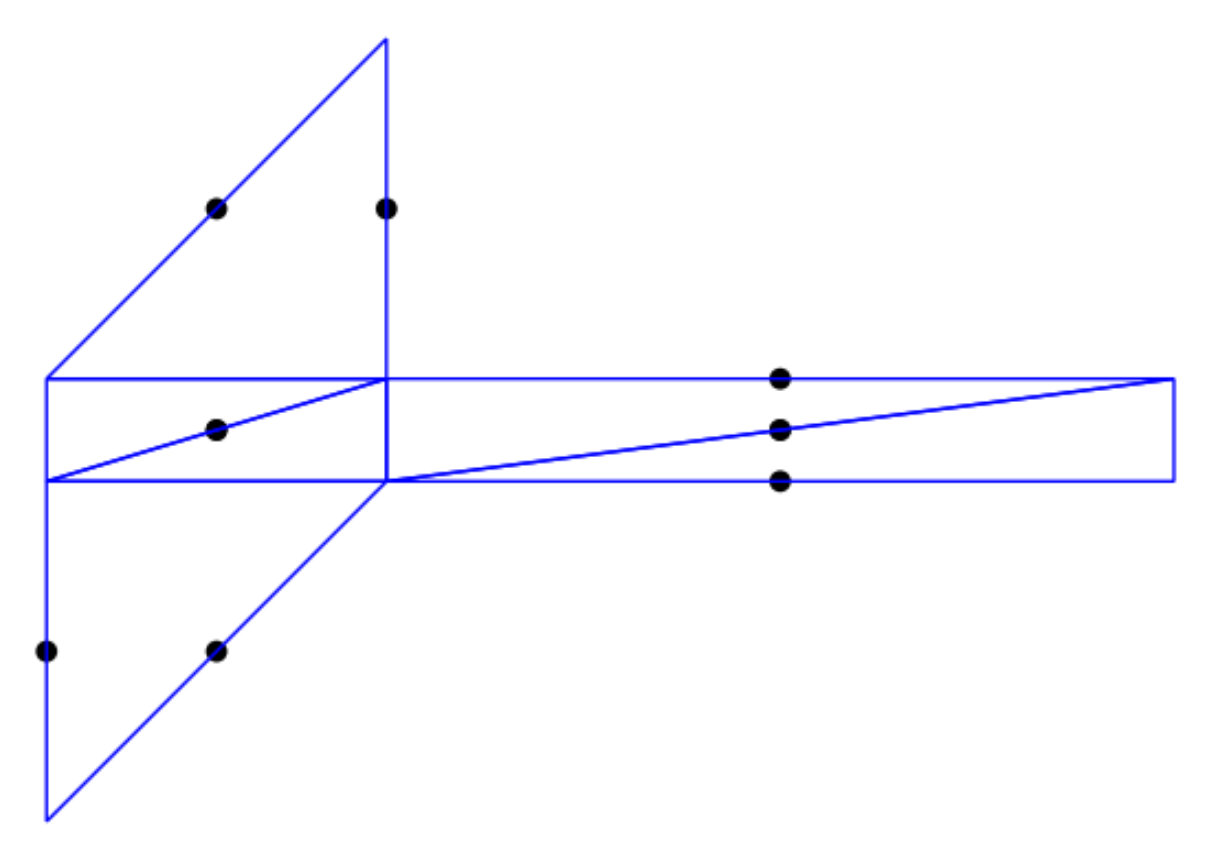}
    \caption{The dots are the computed periodic points on L table with parameter $D = 44$.  Note that the same point may appear twice as opposite parallel sides are identified.}
    \label{figH2pts}
\end{figure}

We begin by applying our algorithm to the $\mathcal{H}(2)$ eigenforms, constructed by McMullen \cite{McMullenEigen} and Calta \cite{Calta}.
These genus 2 translation surfaces arise from certain $L$-shaped polygons, and they are classified by an integer discriminant $D$ congruent to 0 or 1 modulo 4, as well as a spin invariant $\epsilon \in \{-1, 1\}$.
M\"oller showed that the periodic points for these surfaces coincide with the fixed points of their hyperelliptic involutions \cite{Moller2006}.
We tested our algorithm on discriminants through  $D=44$ and returned the correct periodic points. Figure \ref{figH2pts} provides a visual example of output from the algorithm.

\subsection{Prym eigenforms in genus 3}

Each integer $D \geq 8$ satisfying $D \equiv 0, 1$ or $4$ $\mod 8$ determines two S-shaped Euclidean polygons as in \autoref{prympoly}.
Gluing parallel identified sides as shown in the figure gives two genus three translation surfaces $(X^+_D, \omega_D)$ and $(X^-_D, \omega^-_D)$.
The surfaces are known as the \textit{A+ and A- models} of the Weierstrass Prym eigenforms in genus three, and they are one of the known infinite families of Veech surfaces.
We don't give their full description here, but we remark that their underlying Riemann surfaces admit holomorphic involutions known as their \textit{Prym involution}; the Prym involution has 3 fixed points that are not zeros of the 1-forms.
For a complete description, see \cite{LanneauNguyen}.

\begin{figure}[t]
    \centering
    \includegraphics[scale=.3]{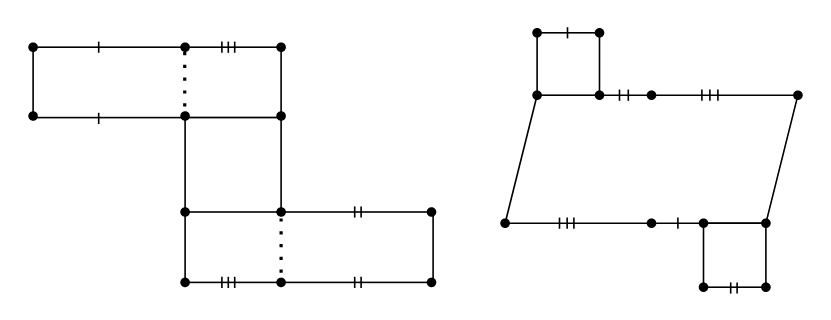}
    \caption{Example form of Model $A+$ surface left, and Model $A-$ surface right.}
    \label{prympoly}
\end{figure}

We now consider the periodic points of genus 3 Prym eigenforms.
For general reasons the fixed points of the Prym involution are periodic, but a priori there could be other periodic points.
We ran our algorithm on the surfaces through discriminant $D=104$ and showed this is not the case, obtaining \autoref{thm:prym}.
This experimental evidence was the motivation and starting input for a full classification by the third author in \cite{Free} that shows that the Prym fixed points are in fact all the periodic points.

\subsection{Example}

We walk through the application of our algorithm to $(X_{17}^+, \omega_{17}^+)$.
We begin by computing the generators of the Veech group of the surface using an implementation of Bowman's algorithm \cite{bowman}.
We find that the generators of $\SL(X^+_{17}, \omega_{17}^+)$ are
\[
\begin{pmatrix}
1 & 2 \\ 0 & 1
\end{pmatrix}, 
\begin{pmatrix}
-1 & 2 \\ \frac{-\sqrt{17}}{4} - \frac{5}{4} & \frac{\sqrt{17}}{2}+\frac{3}{2}
\end{pmatrix},
\begin{pmatrix}
\frac{\sqrt{17}}{2}+\frac{3}{2} & -2 \\ \frac{\sqrt{17}}{4} + \frac{5}{4} & -1
\end{pmatrix},
\begin{pmatrix}
\frac{-3\sqrt{17}}{2}-\frac{13}{2} & \frac{3\sqrt{17}}{2} +\frac{9}{2} \\ \frac{-7\sqrt{17}}{4} - \frac{27}{4} & \frac{3\sqrt{17}}{2}+\frac{11}{2}
\end{pmatrix}
\]
We then determine the constraint lines associated with every triangle in a triangulation of the surface, as detailed in Sections \ref{sec:constraint_reduction} and \ref{sec:finite-set-of-lines}.  In this case, the constraint lines associated with every triangle lie strictly on the boundaries of each triangle.

\begin{figure}
    \centering
    \includegraphics[scale=.15]{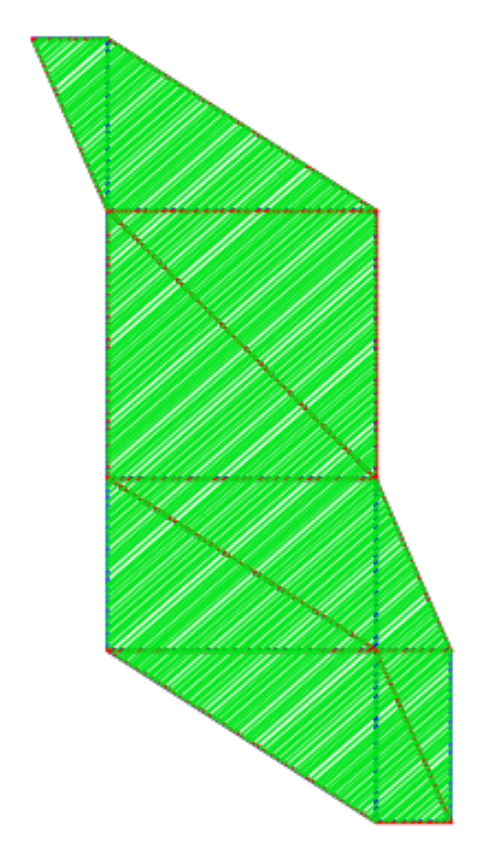}
    \caption{Model $A+$ surface, $D=17$.  The angled lines overlapping the figure are the transformed constraint lines obtained after applying the Veech element $M$.  All the intersection points are candidate periodic points.}
    \label{prymconst}
\end{figure}

We then apply each generator of the Veech group to each constraint line to produce a new set of constraint lines (see in Section \ref{sec:seg_to_points}).
Applying the Veech group element 
\[
M = \begin{pmatrix}
\frac{\sqrt{17}}{2}+\frac{3}{2} & -2 \\ \frac{\sqrt{17}}{4} + \frac{5}{4} & -1
\end{pmatrix}
\begin{pmatrix}
\frac{-3\sqrt{17}}{2}-\frac{13}{2} & \frac{3\sqrt{17}}{2} +\frac{9}{2} \\ \frac{-7\sqrt{17}}{4} - \frac{27}{4} & \frac{3\sqrt{17}}{2}+\frac{11}{2}
\end{pmatrix}
=
\begin{pmatrix}
-2\sqrt{17}-9 & \frac{3\sqrt{17}}{2} + \frac{17}{2} \\ \frac{-7\sqrt{17}}{4} - \frac{31}{4} & \frac{3\sqrt{17}}{2}+\frac{13}{2}
\end{pmatrix}
\]
to the constraint lines produces the transformed constraint lines as pictured in Figure \ref{prymconst}.

\begin{figure}
    \centering
    \begin{minipage}[t]{0.5\textwidth}
        \centering
        \includegraphics[scale=.15]{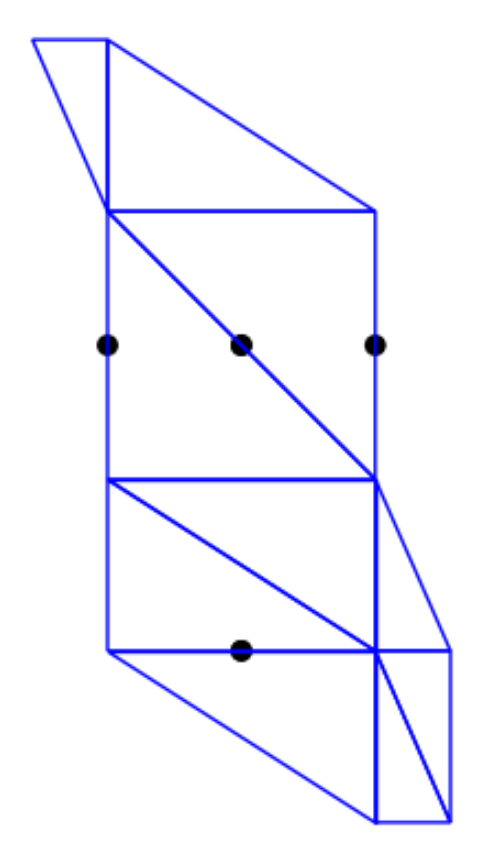}
        \caption{Model $A+$ surface, $D=17$.  The two points on the opposite vertical edges are identified.}
        \label{prym1}
    \end{minipage}\hfill
    \begin{minipage}[t]{0.5\textwidth}
        \centering
        \includegraphics[scale=.15]{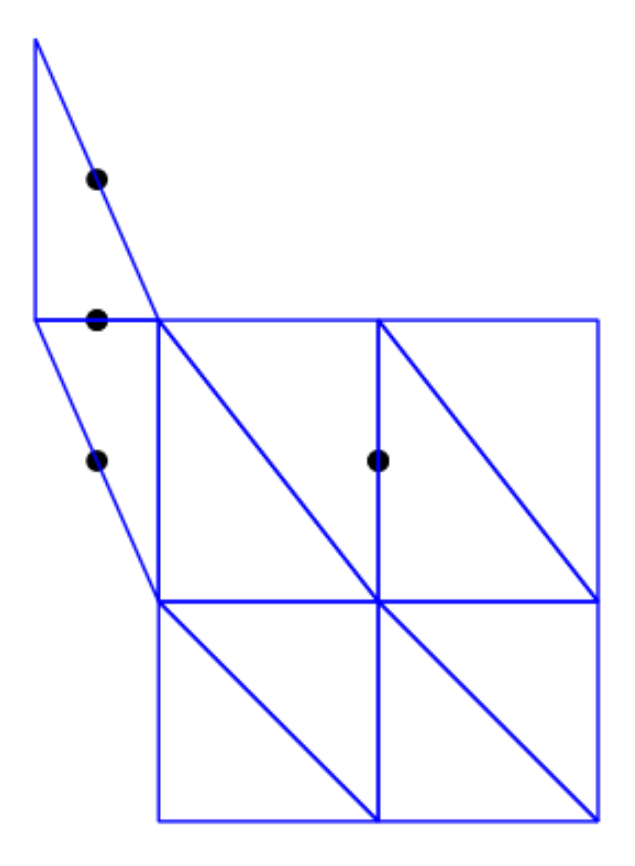}
        \caption{Model $A-$ surface, $D=17$.  The points on opposite edges are identified.}
        \label{prym2}
    \end{minipage}
\end{figure}
After transforming the constraint lines with $M$, we compute the intersection points of the two sets of constraints to obtain a set of 857 candidate periodic points.
Finally, we determine which of the candidate points are periodic by applying the algorithm described in \autoref{lem:reduce-candidate-points} on the set of candidate points.
In this example, after applying the generator
\[
\begin{pmatrix}
\frac{\sqrt{17}}{2}+\frac{3}{2} & -2 \\ \frac{\sqrt{17}}{4} + \frac{5}{4} & -1
\end{pmatrix}
\]
to the set of 857 candidate periodic points, the only points remaining were the three fixed points of the Prym involution points.
The end result is pictured in Figure \ref{prym1}.
Performing a similar analysis for the Model $A-$ surface for $D=17$, we see in Figure \ref{prym2} that the periodic points for this surface are again only the fixed points of the Prym involution.

Our implementation of this algorithm is publicly available at \cite{code}.

\section*{Acknowledgements} We thank the Institute for Computational and Experimental Research in Mathematics (ICERM) for running the Summer@ICERM 2021 REU where this work took place. We also thank Curt McMullen for comments and suggestions on an early draft of this paper. Finally, we are grateful to Paul Apisa for proposing the problem and providing invaluable guidance and support.

\section*{Declarations}
\begin{enumerate}
    \item \textbf{Funding and Competing Interests:} This work took place at the 2021 Summer@ICERM REU, and the authors were funded by ICERM.  But ICERM will neither gain nor lose money upon publication of the manuscript.
    \item \textbf{Disclosure statement:} The authors have no conflicts of interests or competing interests.  The authors have no financial or proprietary interests in any material discussed in this article.  The authors have no competing interests to declare that are relevant to the content of this article.
    \item \textbf{Data availability statement:} The datasets generated during and/or analysed during the current study are available from the corresponding author on reasonable request.
\end{enumerate}

\bibliography{references}

\begin{bibdiv}
\begin{biblist}

\bib{ApisaRtHt}{article}{
      author={Apisa, Paul},
       title={$\uppercase{GL}_2\mathbb{R}$-invariant measures in marked strata:
  generic marked points, {E}arle-{K}ra for strata, and illumination},
        date={2020},
     journal={Geom. Topol.},
      volume={24},
       pages={373\ndash 408},
}

\bib{apisa2020periodic}{misc}{
      author={Apisa, Paul},
      author={Saavedra, Rafael~M.},
      author={Zhang, Christopher},
       title={Periodic points on the regular and double $n$-gon surfaces},
        date={2020},
}

\bib{bowman}{article}{
      author={Bowman, Joshua~P.},
       title={Teichm{\"u}ller geodesics, {D}elaunay triangulations, and {V}eech
  groups},
        date={2008},
     journal={Ramanujan Math. Soc. Lect. Notes},
      volume={10},
       pages={113\ndash 129},
}

\bib{Calta}{article}{
      author={Calta, Kariane},
       title={Veech surfaces and complete periodicity in genus two},
        date={2004},
     journal={J. Amer. Math. Soc.},
      volume={17},
       pages={871\ndash 908},
}

\bib{code}{misc}{
      author={Chowdhury, Zawad},
      author={Everett, Samuel},
      author={Freedman, Sam},
      author={Lee, Destine},
       title={Computing periodic points on {V}eech surfaces},
   publisher={GitHub, \url{https://github.com/SFreedman67/bowman}},
        date={2021},
}

\bib{efw}{article}{
      author={Eskin, Alex},
      author={Filip, Simion},
      author={Wright, Alex},
       title={The algebraic hull of the {K}ontsevich–{Z}orich cocycle},
        date={2018},
     journal={Ann. of Math.},
      volume={188},
       pages={1\ndash 33},
}

\bib{Free}{article}{
      author={Freedman, Sam},
       title={Periodic points of {P}rym eigenforms},
        date={2022},
      eprint={2210.13503},
}

\bib{GutkinHubertSchmidt2003}{article}{
      author={Gutkin, Eugene},
      author={Hubert, Pascal},
      author={Schmidt, Thomas},
       title={Affine diffeomorphisms of translation surfaces: Periodic points,
  {F}uchsian groups, and arithmeticity},
        date={2003-11},
      volume={36},
      number={6},
       pages={847\ndash 866},
         url={https://doi.org/10.1016/j.ansens.2003.05.001},
}

\bib{GutkinJudge}{article}{
      author={Gutkin, Eugene},
      author={Judge, Chris},
       title={Affine mappings of translation surfaces: geometry and
  arithmetic},
        date={2000},
     journal={Duke Mathematical Journal},
      volume={103},
      number={2},
       pages={191 \ndash  213},
         url={https://doi.org/10.1215/S0012-7094-00-10321-3},
}

\bib{Hubert2006}{article}{
      author={Hubert, Pascal},
      author={Lanneau, Erwan},
       title={Veech groups without parabolic elements},
        date={2006-06},
     journal={Duke Mathematical Journal},
      volume={133},
      number={2},
         url={https://doi.org/10.1215/s0012-7094-06-13326-4},
}

\bib{HubertSchmidt}{incollection}{
      author={Hubert, Pascal},
      author={Schmidt, Thomas~A.},
       title={Chapter 6 — {A}n introduction to {V}eech surfaces},
        date={2006},
   booktitle={Handbook of dynamical systems},
      editor={Hasselblatt, B.},
      editor={Katok, A.},
      series={Handbook of Dynamical Systems},
      volume={1},
   publisher={Elsevier Science},
       pages={501\ndash 526},
  url={https://www.sciencedirect.com/science/article/pii/S1874575X06800317},
}

\bib{katok2010fuchsian}{article}{
      author={Katok, Svetlana},
       title={Fuchsian groups, geodesic flows on surfaces of constant negative
  curvature and symbolic coding of geodesics},
        date={2010},
     journal={Homogeneous Flows, Moduli Spaces and Arithmetic},
      volume={10},
       pages={243\ndash 320},
}

\bib{LanneauNguyen}{article}{
      author={Lanneau, Erwan},
      author={Nguyen, Duc-Manh},
       title={Teichm{\"u}ller curves generated by {W}eierstrass {P}rym
  eigenforms in genus three and genus four},
        date={2014},
     journal={J. Topol.},
      volume={7},
       pages={475\ndash 522},
}

\bib{lanneau_finiteness_2017}{article}{
      author={Lanneau, Erwan},
      author={Nguyen, Duc-Manh},
      author={Wright, Alex},
       title={Finiteness of {T}eichm{\"u}ller curves in non-arithmetic rank 1
  orbit closures},
    language={en},
        date={2017},
        ISSN={1080-6377},
     journal={American Journal of Mathematics},
      volume={139},
      number={6},
       pages={1449\ndash 1463},
         url={https://muse.jhu.edu/article/677444},
}

\bib{MasSmi}{article}{
      author={Masur, Howard},
      author={Smillie, John},
       title={Hausdorff dimension of sets of nonergodic measured foliations},
        date={1991},
     journal={Ann. of Math.},
      volume={134},
       pages={455\ndash 543},
}

\bib{MasurTabach}{incollection}{
      author={Masur, Howard},
      author={Tabachnikov, Sergei},
       title={Rational billiards and flat structures},
        date={2002},
   booktitle={Handbook of dynamical systems, vol. 1a},
      editor={Hasselblatt, B.},
      editor={Katok, A.},
   publisher={Elsevier Science B.V.},
     address={Amsterdam},
       pages={1015\ndash 1089},
}

\bib{McMullenEigen}{article}{
      author={McMullen, Curtis~T.},
       title={Billiards and {T}eichm{\"u}ller curves on {H}ilbert modular
  surfaces},
        date={2003},
     journal={J. Amer. Math. Soc.},
      volume={16},
       pages={857\ndash 885},
}

\bib{McMullenPrym}{article}{
      author={McMullen, Curtis~T.},
       title={Prym varieties and {T}eichm{\"u}ller curves},
        date={2006},
     journal={Duke Math. J.},
      volume={133},
       pages={569\ndash 590},
}

\bib{Moller2006}{article}{
      author={M{\"o}ller, Martin},
       title={Periodic points on {V}eech surfaces and the {M}ordell—{W}eil
  group over a {T}eichm{\"u}ller curve},
        date={2006-04},
      volume={165},
      number={3},
       pages={633\ndash 649},
         url={https://doi.org/10.1007/s00222-006-0510-3},
}

\bib{sagemath}{manual}{
      author={{Sage Developers}},
       title={{S}agemath, the {S}age {M}athematics {S}oftware {S}ystem
  ({V}ersion 9.0)},
        date={2020},
        note={{\tt https://www.sagemath.org}},
}

\bib{Shinomiya}{article}{
      author={Shinomiya, Yoshihiko},
       title={Veech surfaces and their periodic points},
        date={2016},
     journal={Conform. Geom. Dyn.},
      volume={20},
       pages={176\ndash 196},
}

\bib{Veech_F_Strs}{article}{
      author={Veech, William~A.},
       title={{Bicuspid {F}-structures and {H}ecke groups}},
        date={201104},
        ISSN={0024-6115},
     journal={Proceedings of the London Mathematical Society},
      volume={103},
      number={4},
       pages={710\ndash 745},
  eprint={https://academic.oup.com/plms/article-pdf/103/4/710/4258173/pdq057.pdf},
         url={https://doi.org/10.1112/plms/pdq057},
}

\bib{wrightsurv}{article}{
      author={Wright, Alex},
       title={Translation surfaces and their orbit closures: An introduction
  for a broad audience},
        date={2015},
     journal={EMS Surv. Math. Sci.},
      volume={2},
       pages={63\ndash 108},
}

\bib{wright2021periodic}{article}{
      author={Wright, Benjamin},
       title={Periodic points of {W}ard-{V}eech surfaces},
        date={2021},
      eprint={2106.09116},
}

\bib{zorich}{incollection}{
      author={Zorich, Anton},
       title={Flat surfaces},
        date={2006},
   booktitle={Frontiers in number theory, physics, and geometry vol.i},
      editor={Cartier, P.},
      editor={Julia, B.},
      editor={Moussa, P.},
      editor={Vanhove, P.},
   publisher={Springer Verlag Berlin Heidelberg},
     address={Berlin},
       pages={439\ndash 586},
}

\end{biblist}
\end{bibdiv}
\bibliographystyle{abbrv}

\end{document}